\documentclass[12pt]{amsart}
\usepackage[colorlinks, linkcolor=blue, anchorcolor=blue, citecolor=blue]{hyperref}
\usepackage{amsmath}
\usepackage{amsfonts}
\usepackage{amsthm}
\usepackage{geometry}
\usepackage{mathrsfs}
\usepackage{amssymb}
\usepackage[percent]{overpic}
\usepackage{bm}
\usepackage[all]{xy}
\usepackage{graphicx}
\usepackage{subfigure}
\usepackage{latexsym}
\usepackage{epigraph}
\usepackage{fancyhdr}
\usepackage{tikz}
\usepackage{color}
\usepackage{mathtools}

\newcommand*\circled[1]{\tikz[baseline=(char.base)]{
            \node[shape=circle,draw,inner sep=0.8pt] (char) {#1};}}

\numberwithin{equation}{section}
\newtheorem*{obser}{Observation}
\newtheorem*{tuc}{Topological uniformness condition}

\newtheorem{definition}{Definition}[section]

\newtheorem{theorem}{Theorem}[section]
\newtheorem{corollary}{Corollary}[section]
\newtheorem{prop}{Proposition}[section]

\theoremstyle{remark}
\newtheorem{remark}{Remark}[section]
\newtheorem*{ack}{Acknowledgement}

\def\mod{\operatorname{mod}}
\def\sing{\operatorname{Sing}}

\def\dist{\operatorname{dist}}

\def\hc{\operatorname{\widehat{\mathbb C}}}
\def\C{\operatorname{\mathbb C}}
\def\ud{\operatorname{\mathbb D}}

\def\s{\operatorname{\mathcal{S}}}
\def\sing{\operatorname{Sing}}

\def\k{\operatorname{\mathcal{K}}}

\begin{document}
\title{Entire functions arising from trees}
\author{Weiwei Cui}
\date{}

\address{Mathematiches Seminar, Christian-Albrechts-Universit\"at zu Kiel, Ludewig-Meyn-Str. 4, 24098 Kiel, Germany.}
\email{cui@math.uni-kiel.de}


\maketitle

\begin{abstract}
Given any infinite tree in the plane satisfying certain topological conditions, we construct an entire function $f$ with only two critical values $\pm 1$ and no asymptotic values such that $f^{-1}([-1,1])$ is ambiently homeomorphic to the given tree. This can be viewed as a generalization of a result of Grothendieck to the case of infinite trees. Moreover, a similar idea leads to a new proof of a result of Nevanlinna and Elfving.
\end{abstract}

\section{Introduction}

An entire function is called a \emph{Shabat} entire function if it has only two critical values and no asymptotic values. In case that it is a polynomial, it is called a Shabat polynomial. Without loss of generality, we may assume that the two critical values are $\pm 1$. For a Shabat entire function $f$, we define $T_f:=f^{-1}([-1,1])$, which can be considered as a graph embedded in $\C$ whose vertices are the preimages of $\pm 1$ and whose edges are the preimages of $(-1,1)$. For instance, the function $z\mapsto\sin(z)$ is a Shabat entire function. The following observation for Shabat entire functions is clear.

\begin{obser}
Let $f$ be a Shabat entire function. Then $T_f$ is a tree. Moreover,
\begin{itemize}
\item[$(1)$] If $f$ is a polynomial, then $T_f$ is a finite tree.
\item[$(2)$] If $f$ is transcendental, then $T_f$ is an infinite tree.
\end{itemize}
\end{obser}

We say that two trees $T_1$ and $T_2$ are \emph{equivalent}, if there is a homeomorphism $\varphi: \C\to\C$ such that $\varphi(T_1)=T_2$ with vertices and edges being mapped to vertices and edges respectively. In this sense, we also say that $T_1$ is ambiently homeomorphic to $T_2$. If $T=T_f$ for some Shabat entire function $f$, then $T$ is called a \emph{true tree}.

An inverse problem to the above observation asks for the realization of Shabat entire functions from any given trees embedded in the plane, or equivalently, finding a true tree which is equivalent to the given one. If such a true tree exists, we say that it is a \emph{true form} of the given tree. If a tree is finite, the inverse problem has a positive solution which is essentially due to Grothendieck, see \cite{shabat1} and \cite{girondo} for instance. That is to say, each finite tree $T$ has a true form. This is an important aspect of Grothendieck's theory of dessins d'enfants. For other aspects of finite true trees, see \cite{bishop4}.
However, in case that a tree is infinite, it is possible that no corresponding true tree exists (see Corollary \ref{cor21} and the examples in Section 5).

In this paper we shall explore some topological conditions on the tree that ensure the existence of true forms. As discussed above, the existence of true forms is equivalent to the existence of certain Shabat entire functions. Therefore, essentially we need to construct Shabat entire functions based on the information given by the tree.

Our work is motivated by Bishop's recent groundbreaking technique of quasiconformal folding \cite{bishop1}. More precisely, he proved that for any infinite tree $T$ in the plane satisfying a certain \emph{bounded geometry} condition, there is always a Shabat entire function $f$ such that $T_{f}$ approximates $T$ in the sense that $T_f$ is the quasiconformal image of a tree $T'$ obtained by adding some branches to $T$. In general, $T_f$ is not equivalent to $T$, but $T_f$ can be chosen to lie in a small neighborhood of $T$.

For some basic knowledge on graph theory, we refer to \cite{diestel}. To state our conditions, we first introduce some definitions.

\begin{definition}[Semi- and bi-infinite path]\label{path}
Let $T$ be an infinite tree in the plane. An bi-infinite path $\gamma$ in $T$ is a homogeneous tree in $T$ of valence two. A semi-infinite path is a tree of local valence two at every vertex except for one at which the local valence is one.
\end{definition}

\begin{definition}[Kernel]\label{ker}
Let $T$ be an infinite tree in the plane. A \emph{kernel} of $T$, denoted by $\mathcal{K}(T)$, is defined as a semi-infinite path if $T$ has only one complementary component in the plane, or as the union of all bi-infinite paths in $T$ otherwise.
\end{definition}

\begin{figure}[htbp] 
    \begin{minipage}[t]{0.5\linewidth}
    \centering
    \includegraphics[width=6cm]{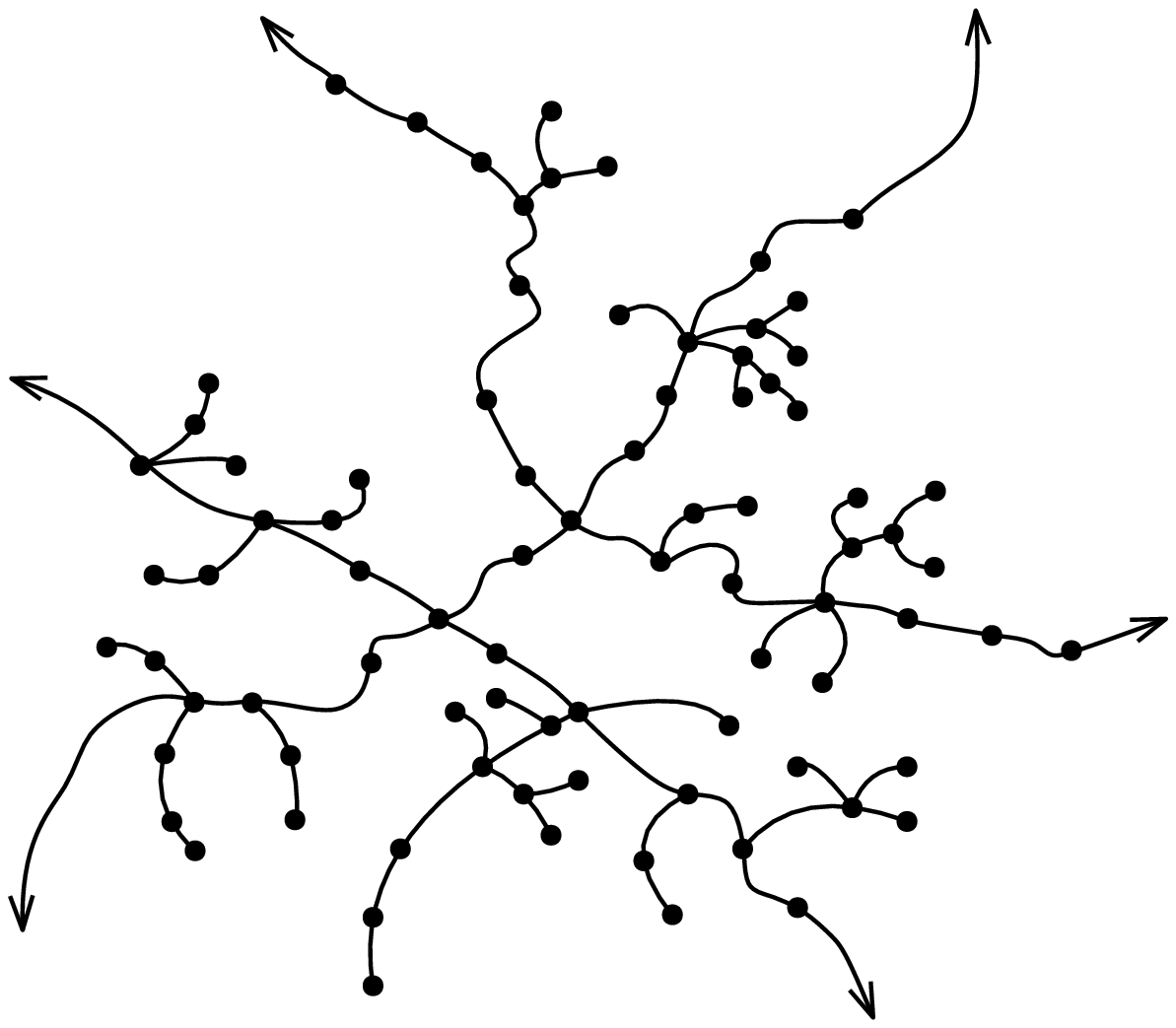}
    \end{minipage}%
 \begin{minipage}[t]{0.5\linewidth}
    \centering
   \includegraphics[width=6cm]{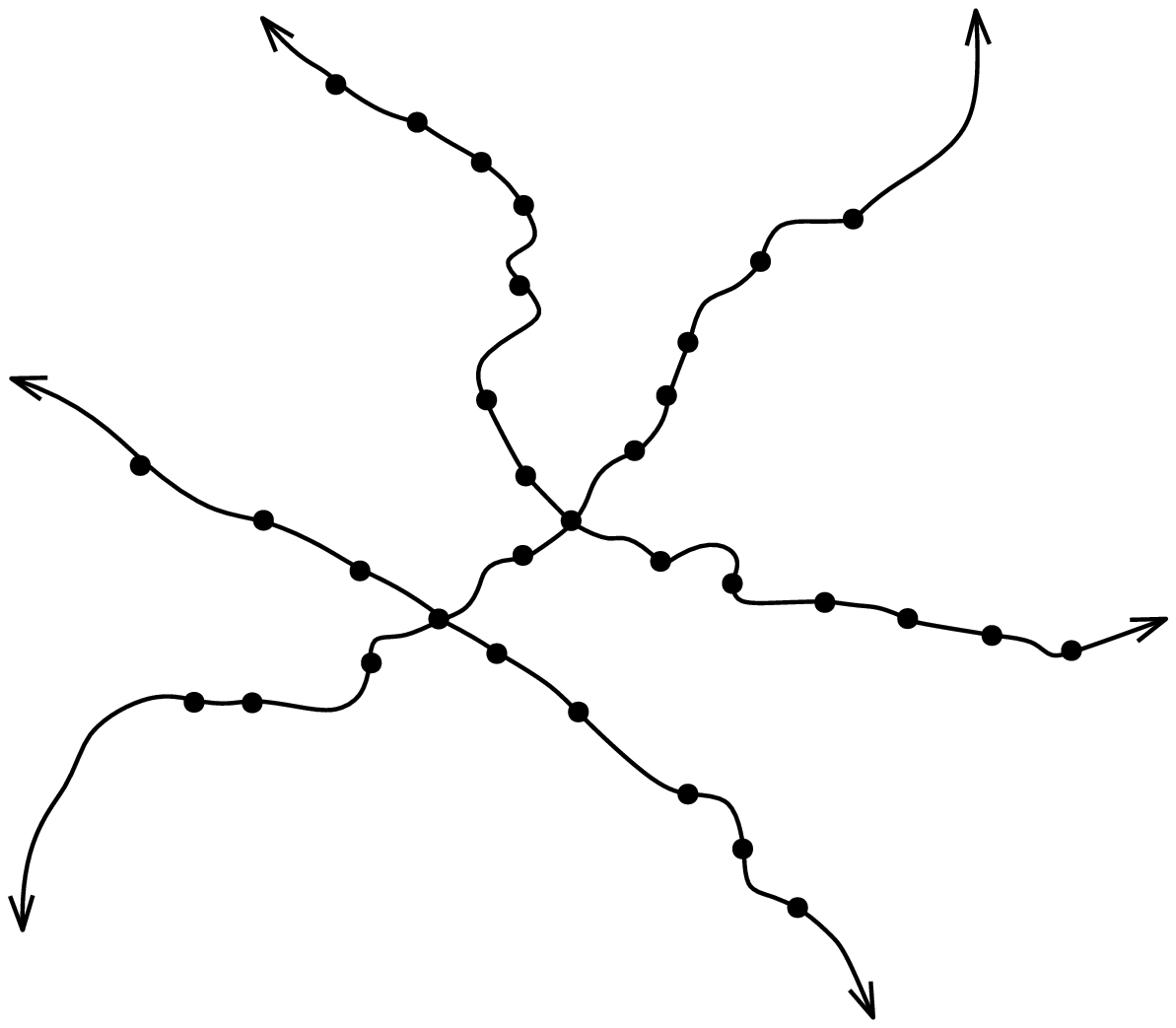} 
    \end{minipage}
    \caption{Left: An infinite tree with more than one complementary component. Right: The kernel of the left-hand side tree.}
    \label{figure1}
\end{figure}

\begin{figure}[htbp]
    \begin{minipage}[t]{0.5\linewidth}
    \centering
    \includegraphics[width=6cm]{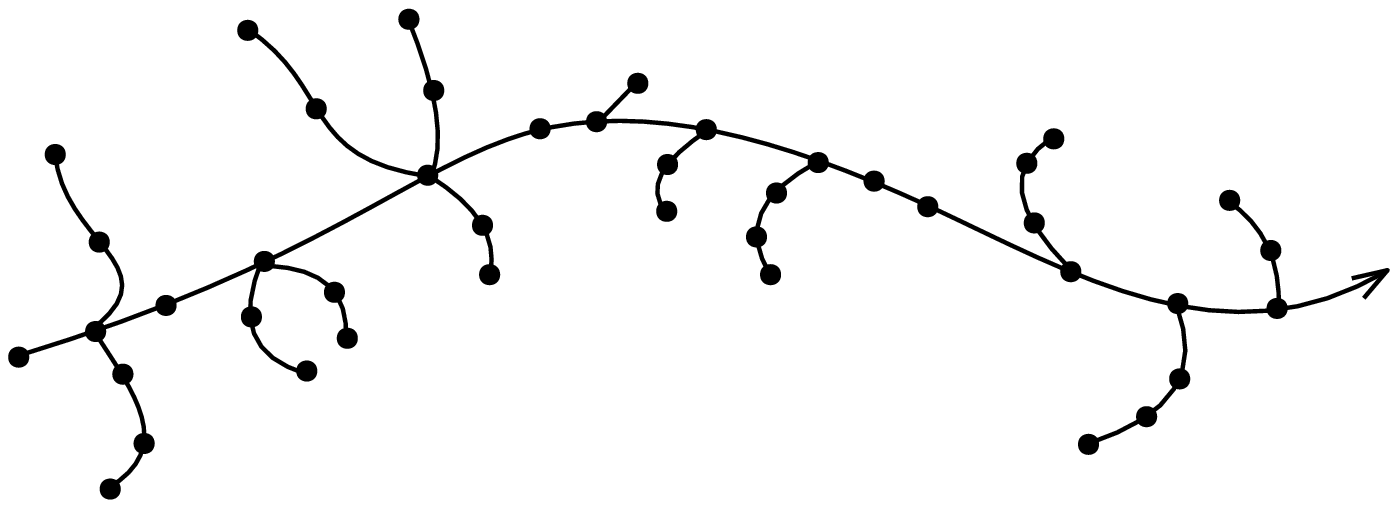}
    \end{minipage}%
 \begin{minipage}[t]{0.5\linewidth}
    \centering
   \includegraphics[width=6cm]{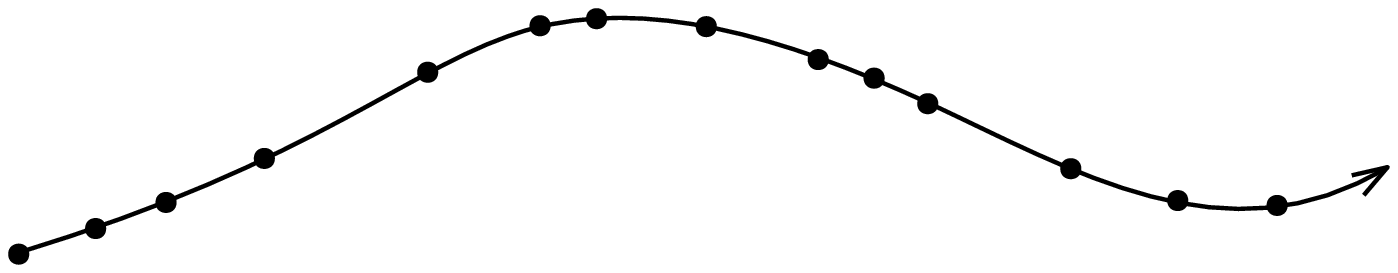}
    \end{minipage}
    \caption{An infinite tree with only one complementary component and one of its kernel.}
    \label{figure2}
\end{figure}

Roughly speaking, the kernel of a tree is obtained by cutting off all finite branches attached on the tree. See Figure \ref{figure1}. It is easy to see that a homogeneous tree coincides with its kernel. The notion of kernel is well defined in the sense that if an infinite tree $T$ has finitely many ($\geq 2$) complementary components in the plane, then the kernel $\k(T)$ is unique. In case that the tree has exactly one unbounded complementary component (which indeed happens for some Shabat entire functions, such as $z\mapsto\cos\sqrt{z}$), there is no bi-infinite path. In this situation, we consider a path which connects any \emph{fixed} finite vertex to $\infty$ without backtracking in the tree instead. See Figure \ref{figure2}.

An infinite tree (or more generally a connected graph) can be viewed as a metric space. To be more precise, a graph $G$ is a pair $(V, E)$, where $V$ is a set of vertices and $E$ is a set of pairs of elements in $V$ which are called edges. Then the length of a path is defined to be the number of edges, and the distance of two vertices is the length of the shortest path connecting them. We call this metric the \emph{word metric} of the tree. Thus we can speak of \emph{word distance} on the tree. Moreover, endowed with this metric any connected tree will be a geodesic metric space. In the following, we will use $d(A,B)$ to denote the word distance between two vertices or two subsets or a vertex and a set.

Now we can formulate our conditions as follows.

\begin{tuc}[TUC]\label{tuc}
Let $T$ be an infinite  tree in the plane, satisfying the following conditions:
\begin{itemize}
\item[$(T1)$] the local valence of the tree is uniformly bounded above;
\item[$(T2)$] $T$ has finitely many complementary components in the plane;
\item[$(T3)$] $d(v, \k(T))$ is uniformly bounded above; that is, there exists $M\in\mathbb{N}$ such that $d(v, \k(T))\leq M$ for all vertices $v$ of $T$.
\end{itemize}
\end{tuc}

Our main result can be stated as follows.

\begin{theorem}[Entire functions from trees]\label{rea}
Let $T$ satisfy the topological uniformness condition. Then $T$ has a true form.
\end{theorem}

In other words, the above theorem says that there is a Shabat entire function $f$ such that $T_f$ is equivalent to $T$. The next question one may ask is the \emph{uniqueness} of the true form. This follows from the fact that two entire functions with only two singular values which are topologically equivalent in the sense of Eremenko and Lyubich \cite{eremenko2} are in fact conformally equivalent \cite[Proposition 2.3]{epstein5}. Thus we see that the true form is unique up to affine maps in the plane.

We note that no condition in the TUC can be dropped. We will construct various examples to show this in Section 5. These examples also indicate that one can slightly extend the condition we imposed above.

\medskip
\noindent{\emph{Idea of proof.}} The proof of Theorem \ref{rea} is based on the criterion of the classical type problem from the geometric theory of meromorphic functions (see Section 2.1). To each such tree $T$ in the theorem, one can construct a Speiser graph $\Gamma$. Then it is well known that there is a surface $(X,p)$ spread over the sphere of class $\s$ which corresponds one-to-one to the Speiser graph $\Gamma$. If the surface is parabolic, then there is a Shabat entire function $f$ whose Speiser graph is $\Gamma$. Moreover, it will be clear from the proof that $T_f$ will be equivalent to $T$. We will show that as long as the topological uniformness condition is satisfied, the surface will be of parabolic type and hence ensures the existence of a Shabat entire function.

\medskip
\noindent{\emph{Structure of the article.}} Section 2 collects some notions and results required. Then the proof of Theorem \ref{rea} is given in Section 3. In the following section, we give a new proof of a result of Nevanlinna and Elfving. The last section is devoted to the construction of various examples showing the sharpness of our theorem.

\medskip
\begin{ack}
I would like to thank Walter Bergweiler, Alexandre Eremenko, Sergiy Merenkov and Lasse Rempe-Gillen for many useful discussions. 
\end{ack}

\section{Preliminaries}

\subsection{The type problem}

We give here a brief introduction to the \emph{type problem} in the geometric function theory. For a detailed description, we refer to \cite{eremenko5} and \cite{drasin1}.

According to the Uniformization Theorem, every open simply connected Riemann surface $X$ is conformally equivalent to either the complex plane $\C$ or the unit disk $\ud$. The Riemann surface is said to be of \emph{parabolic} or \emph{hyperbolic} type, respectively. In other words, there exists a conformal map $\phi: X_0 \to X$, where $X_0$ is $\C$ or $\ud$. The map $\phi$ is called the \emph{uniformizing map}.

A \emph{surface spread over the sphere} is an equivalence class $[(X,p)]$, where $X$ is an open, simply connected, topological surface, and $p: X\to\hc$ a topologically holomorphic map, that is, an open, continuous and discrete map. Here $(X_1, p_1)$ is equivalent to $(X_2, p_2)$ if there exists a homeomorphism $\psi: X_1\to X_2$ such that $p_1=p_2\circ\psi$.

A theorem of Sto\"{\i}low says that any topologically holomorphic map is locally modelled by the power map $z\mapsto z^k$ for some $k\in\mathbb{N}$ \cite{stoilow1}. Thus there is a unique Riemann surface structure on $X$ which makes it a Riemann surface. So there is a uniformizing map $\phi: X_0 \to X$, and hence $f:=p\circ\phi$ is a meromorphic function in $\C$ or $\ud$. The surface $(X,p)$ is called the surface associated to $f$. For a given surface spread over the sphere, the type problem is the determination of the conformal type of the surface.

Some classical results can be viewed as criteria of the type problem, for instance, Picard's theorem, Ahlfors' five islands theorem and others. In particular, we recall the following result due to Nevanlinna, which can also serve as a criterion for the type problem. Recall that a point $a\in\hc$ is said to be a \emph{totally ramified value} (of multiplicity $m\geq 2$) of a surface $(X,p)$ spread over the sphere, if all preimages of $a$ under $p$ are multiple (of multiplicity at least $m$). Take $m=\infty$ if $a$ is omitted. The following result is due to Nevanlinna.

\begin{theorem}\label{nevan}
If a surface spread over the sphere has $q$ totally ramified values of multiplicities $m_k$ for $1\leq k\leq q$, and
\begin{equation}
\sum_{k=1}^{q}\left(1-\frac{1}{m_k} \right)>2,
\end{equation}
then the surface is hyperbolic. In particular, a parabolic surface has at most four totally ramified values.
\end{theorem}

We mentioned in the introduction that some trees do not have true forms. Indeed, the above theorem of Nevanlinna can be used to construct such trees. We only mention the following example.

\begin{corollary}\label{cor21}
Any homogeneous tree of valence $\geq 3$ does not have a true from.
\end{corollary}

\begin{proof}
Suppose that such a tree indeed gives a Shabat entire function $f$. Then $\pm 1$ will be two totally ramified values (another one is $\infty$) of $f$. Every vertex of the tree is a preimage of one of $\pm 1$. Hence for $q=3$,  $m_1=m_2=3$ and $m_3=\infty$, the inequality in the above theorem is satisfied, which is a contradiction since the surface corresponding to a transcendental entire function is always parabolic.
\end{proof}

\subsection{The Speiser class}

For a surface $(X,p)$ spread over the sphere, the map $p: X\to\hc$ is modelled on the power map $z\mapsto z^k$ in a neighborhood of each point $x\in X$ for some $k:=k(z)\in\mathbb{N}$. If $k\geq 2$, then $x$ is a critical point of $p$, and $p(x)$ is called a \emph{critical value} of $p$. A point $b\in\hc$ is called an \emph{asymptotic value} of $p$ if there exists a curve $\gamma:[0,\infty)\to X$, such that $\gamma(t)$ leaves every compact subset of $X$ as $t\to \infty$, and $p(\gamma(t))\to b$ as $t\to \infty$. When we say $a$ is a \emph{singular value}, we always mean it is a critical value or an asymptotic value.

A surface $(X,p)$ spread over the sphere belongs to the \emph{Speiser class}\footnote{We note here that the terminology \emph{Speiser class} is also used in other setting where a function meromorphic in the plane belongs to the Speiser class if the function has finitely many singular values. But here the setting is larger in the sense that every Speiser class function gives a parabolic surface spread over the sphere, but a Speiser class surface need not be parabolic.}, denoted by $\s$, if $p$ has only finitely many singular values. Equivalently, we also say that the set of singular values of $(X,p)$ is the smallest closed subset $\sing$ of $\hc$ such that
$$p: X\setminus\{p^{-1}(\sing)\}\to\hc\setminus\sing$$
is a covering map. A surface in class $\s$ has a combinatorial representation in terms of a \emph{Speiser graph} (also called \emph{line complex}). To give a definition of a Speiser graph, assume that $(X,p)\in\s$ has $q$ singular values $a_1, \dots, a_q$. Then one can choose an oriented Jordan curve $L$ on $\hc$ passing through $a_1, a_2, \dots, a_q$ in cyclic order. The curve $L$ decomposes $\hc$ into two components $A$ and $B$ such that $A$ is to the left-hand side of positive orientation of $L$. Then we choose two base points $\circ\in A$ and $\times\in B$ and connect them by open Jordan arcs $\gamma_j$ ($j=1,\dots,q$) such that $\gamma_j\cap\gamma_k = \emptyset$ for $j\neq k$ and $\gamma_j$ intersects the segment $(a_j, a_{j+1})$ in $L$ in exactly one point with indices modulo $q$. Then the Speiser graph $\Gamma$ of the surface $(X,p)$ is defined to be the embedding of $p^{-1}(\bigcup_j \gamma_j)$ into the plane $\C$. For some figures of Speiser graphs, see \cite[Chapter 7, Section 4]{goldbergmero} and also the figures in Section 5.

We consider Speiser graphs defined above as graphs in the sense that preimages of $\circ$ (marked as circles) and preiamges of $\times$ (marked as crosses) serve as vertices of the graph while the set of edges contains all preimages of $\gamma_j$ for all $j$. Therefore, Speiser graphs have the following properties: Each edge connects a preimage of $\circ$ and a preimage of $\times$; around each vertex there are exactly $q$ edges emanating from this vertex. Moreover, the faces around a vertex have certain cyclic order inherited from the orientation of $L$ (this can be seen by marking faces using $a_j$). For a detailed description of graphic properties of Speiser graphs we refer to \cite[Chapter 7]{goldbergmero}, in which one can find applications of Speiser graphs in the theory of meromorphic functions. An important property is that, up to a choice of a base curve, there is a one-to-one correspondence between surfaces in class $\s$ and equivalence classes of Speiser graphs. Here two Speiser graphs are equivalent if they are ambiently homeomorphic. For a detailed proof of this fact, we refer to \cite{goldbergmero}.

\begin{remark}
Transcendental entire functions in class $\s$ attract a lot of interest in transcendental dynamics since the work of Eremenko and Lyubich \cite{eremenko2}, mainly due to the fact that these functions have certain similarities to polynomials from a dynamical point of view.
\end{remark}

\subsection{The Doyle-Merenkov criterion}

There are many criteria to determine the conformal type of a surface in class $\s$ using Speiser graphs by Ahlfors, Nevanlinna, Ullrich, Wittich and others. For a detailed account, see \cite{volkovyskii}. Doyle used a modification of the Speiser graph to give a necessary and sufficient condition \cite{doyle1}. This was generalized to a more general setting by Merenkov by using a geometric method, see \cite{merenkov3}.

We say that an infinite, locally finite, connected graph is parabolic/hyperbolic, if the simple random walk on the graph is recurrent/transient. The Doyle-Merenkov criterion connects the type of a surface and that of some extended Speiser graph which is defined as follows. 

Let $\Lambda$ denote a half-plane lattice with vertices at $\mathbb{Z}\times\mathbb{N}_0$, and $\Lambda_n:=\Lambda/n\mathbb{Z}$ a half-cylinder lattice. Let $\Gamma$ be a Speiser graph. Then the \emph{extended Speiser graph} $\Gamma_n$ is defined as follows: To each unbounded face we embed a half-plane lattice $\Lambda$ by identifying corresponding edges and vertices on the boundaries, and to each face of $\Gamma$ with $2k$ edges on the boundary for $k\geq n$, we embed the half-cylinder lattice $\Lambda_{2k}$, The graph $\Gamma_n$ obtained in this way is called an \emph{extended Speiser graph} of $\Gamma$. The graph obtained by adding a half-plane lattice to each unbounded face and leaving all bounded faces unchanged is denoted by $\Gamma_{\infty}$. As a matter of fact, $\Gamma_{\infty}=\Gamma_n$ for some $n\in\mathbb{N}$ if all bounded faces of $\Gamma$ have an upper bound for the numbers of edges on the boundaries.

\begin{theorem}[Doyle-Merenkov criterion]\label{dmc}
Let $n\in\mathbb{N}$ be fixed. A surface spread over the sphere $(X,p)\in\s$ is parabolic if and only if $\Gamma_n$ is parabolic.
\end{theorem}

This criterion is easy to apply and has many applications, one of which is the construction of a parabolic surface with negative mean excess, see \cite{merenkov2}. This answers a question of Nevanlinna in negative.

\subsection{Quasi-isometry}

We shall use in the proof of Theorem \ref{rea} the notion of quasi-isometry between two infinite, locally finite graphs both endowed with the word metric.

\begin{definition}[Quasi-isometry]\label{qi}
Let $(X_1, d_1)$ and $(X_2, d_2)$ be two metric spaces. A map $\Phi: X_1 \to X_2$ is called a quasi-isometry, if it satisfies the following two conditions:
\begin{itemize}
\item[$(1)$] for some $\varepsilon>0$, the $\varepsilon$-neighborhood of the image of $\Phi$ in $X_2$ covers $X_2$;
\item[$(2)$] there are constants $k\geq 1$, $C\geq 0$ such that for all $x_1, x_2\in X_1$,
$$\frac{1}{k}\cdot d_{1}(x_1, x_2)-C\leq d_2 (\Phi(x_1), \Phi(x_2)) \leq k\cdot d_1(x_1, x_2)+C.$$
\end{itemize}
\end{definition}

It is easy to check that being quasi-isometric is an equivalence relation. The notion of quasi-isometry was introduced by Gromov \cite{gromov1} and Kanai \cite{kanai2}. Many properties of metric spaces are preserved under quasi-isometries. For instance, the Gromov hyperbolicity of geodesic metric spaces is stable under quasi-isometric maps. For us, a particular important property of quasi-isometry is that the type of an infinite, locally finite graph is stable under quasi-isometries. This is essentially due to Kanai, see \cite[Corollary 7]{kanai1}. 

\begin{theorem}[Stability of the type]\label{stable}
Let $\Gamma_1$ and $\Gamma_2$ be two connected, finite valence graphs endowed with the word metric, which are quasi-isometric. Then $\Gamma_1$ and $\Gamma_2$ are simultaneously hyperbolic or parabolic.
\end{theorem}

\subsection{Combinatorial modulus}

To determine the type of a graph, we shall use the notion of combinatorial modulus, which can be viewed as a discrete version of the classical conformal modulus. The notion was used by Duffin \cite{duffin1} and Cannon \cite{cannon1}. Here we will follow the presentation in \cite[Section 5]{merenkov3}. 

A \emph{mass distribution} for a graph $\Gamma$ is a non-negative function on the set $E$ of edges of $\Gamma$. A chain in $\Gamma$ is a sequence $\{e_j\}_{j=M}^{N}$ with $-\infty\leq M \leq N \leq \infty$, where $e_j$ are edges and adjacent edges share a vertex. Let $\mathcal{C}$ be a family of chains in $\Gamma$. A mass distribution $m$ is admissible for the family $\mathcal{C}$, if $\sum m(e_j)\geq 1$ for each chain $\{e_j\}$ in $\mathcal{C}$. Then the \emph{combinatorial modulus} of the chain family $\mathcal{C}$ is defined as
$$\mod\mathcal{C}:=\inf\left\{\sum_{e\in E} m(e)^2  \right\},$$
where the infimum is taken all admissible mass distributions. The reciprocal of the combinatorial modulus is called the \emph{extremal length} of $\mathcal{C}$, denoted by $\lambda(\mathcal{C})$. The combinatorial modulus has properties similar to those of the conformal modulus (cf. \cite{ahlfors9}, \cite{soardi}).

A domain in a graph is a connected subset of the graph in the sense that every two vertices in the domain can be connected by a chain whose edges belong to this domain. An annulus in a graph is a subset of the set of edges whose complement in the graph has two disjoint domains. A sequence of annuli $\{A_n\}$ is nested if the annuli are pairwise disjoint, and $A_{n+1}$ separates $A_n$ from $\infty$. We will use $\mod A$ for an annulus $A$ to denote the combinatorial modulus of the family of chains connecting the two boundary components of $A$.

A locally finite graph $\Gamma$ is hyperbolic or parabolic if and only if $\lambda(\Gamma_v)$ is finite or infinite respectively for some vertex $v$ (and hence for every vertex) in $\Gamma$, where $\Gamma_v$ denotes the family of chains connecting $v$ to infinity. See, \cite[Corollary 3.84]{soardi}, for instance. Compare this with a similar result for the conformal modulus, see for instance \cite[Theorem 1]{drasin1}.

For our applications, to determine parabolicity it is sufficient to use the following

\begin{prop}\label{howto}
Let $\Gamma$ be a connected finite valence graph. If there exists a sequence $\{A_n\}$ of disjoint nested annuli with
$$\sum_{n=1}^{\infty} \frac{1}{\mod A_n}=\infty,$$
then $\Gamma$ is parabolic.
\end{prop}

This result will be used later to prove the parabolicity of surfaces associated to trees satisfying the TUC. As we cannot locate a proof of this result, we will include here a short proof.

\begin{proof}
As noted above, to show that $\Gamma$ is parabolic we can choose any vertex $v$ and show that $\lambda(\Gamma_v)=\infty$, where $\Gamma_v$ denotes the family of chains connecting $v$ to infinity. Here we choose $v$ to be a vertex belonging to the inner complementary component of $A_1$. Then
$$\lambda(\Gamma_v)\geq \sum_{n=1}^{\infty} \lambda( A_n)=\sum_{n=1}^{\infty} \frac{1}{\mod A_n}=\infty.$$
The first inequality above follows from the definition of the extremal length.
\end{proof}

\section{TUC implies parabolicity}

The proof of Theorem \ref{rea} is divided into two steps. First we construct a Speiser graph $\Gamma$ considered as the dual graph of a triangulation of the plane induced from the given tree $T$. This is a standard procedure. In the next step, we use the Doyle-Merenkov criterion to show that some extended Speiser graph $\Gamma_{\infty}$ is of parabolic type, and thus the surface corresponding to $\Gamma$ is parabolic. With this in hand we will obtain a Shabat entire function.

\subsection{Construction of a Speiser graph from the tree}

Suppose that $T$ is any given tree in the plane which satisfies the topological uniformness condition. Then $T$ can be colored in a bipartite pattern (simply choosing one vertex as black and every vertex which has word distance one as white, and then proceed). Connect each side of an edge to infinity by adding two curves at the two endpoints. This gives a triangulation of the sphere based on the given tree. See Figure \ref{triangle}. Then we consider the dual graph of this triangulation, denoted by $\Gamma$ (that is, choose a point in each triangle and then connect two points as long as the two triangles share a common edge). It is easy to see that $\Gamma$ is a Speiser graph. Thus there is a surface $(X,p)$ spread over the sphere in class $\s$ which corresponds to $\Gamma$.
\begin{figure}[htbp] 
    \centering
    \includegraphics[width=11cm]{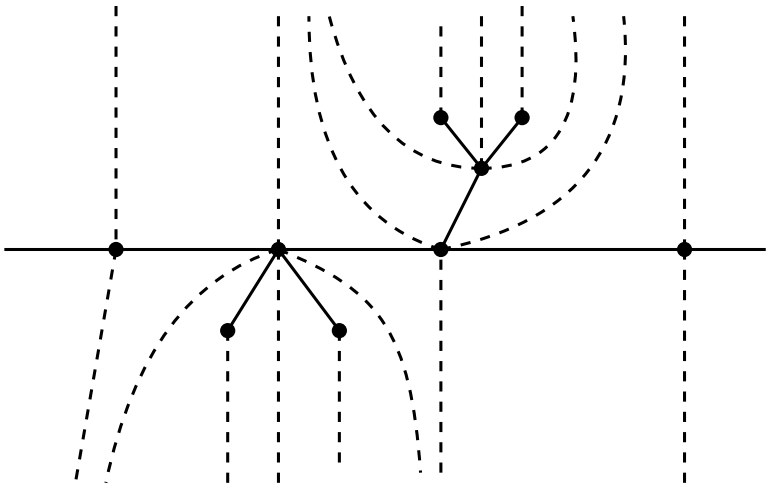}
    \caption{A triangulation of the plane based on a given tree. Every triangle has one side of an edge on the tree as one edge. The dual graph of this triangulation will be the Speiser graph associated to this tree.}
    \label{triangle}
\end{figure}

Since we start with an infinite tree in the plane with finitely many complementary components, the Speiser graph $\Gamma$ will have finitely many unbounded faces, and every other face has a uniformly bounded number of edges on its boundary.

For later purposes, we also consider the same process for the kernel $\k(T)$ of $T$, and thus obtain another Speiser graph $\Sigma$. The surface corresponding to $\Sigma$ will be parabolic. Note that the kernel of a tree satisfying the TUC also satisfies the TUC. If we knew Theorem \ref{rea} already, we could conclude that the surface of the kernel is parabolic. Instead, we will prove this fact by using Goldberg's theory of almost periodic ends. More precisely, except for a finite portion in the corresponding Speiser graph every component of the rest is a sine-end (which is a special case of periodic ends). Since there are only finitely many such ends, then the corresponding surface is parabolic (see \cite[Chapter 7, Section 7]{goldbergmero}). Thus the existence of Shabat entire functions is assured.

\subsection{Determination of the type}

Now we consider an extended Speiser graph by embedding a half-plane lattice in each unbounded face of $\Gamma$ to get an extended Speiser graph $\Gamma_{\infty}$, which is infinite and locally finite. We also denote by $\Gamma_{\infty}^{*}$ the dual graph of $\Gamma_{\infty}$.

Similarly, the Speiser graph $\Sigma$ corresponding to the kernel $\k(T)$ is used to obtain an extended Speiser graph $\Sigma_{\infty}$ and denote by $\Sigma_{\infty}^{*}$ its dual.

As can be observed from the construction, $\Gamma_{\infty}^{*}$ is nothing but topologically an extension of the tree $T$ by embedding a half-plane lattice in each complementary component; so is $\Sigma_{\infty}^{*}$.

\begin{prop}\label{equ1}
$\Gamma_{\infty}^{*}$ is quasi-isometric to $\Gamma_{\infty}$ and $\Sigma_{\infty}^{*}$ is quasi-isometric to $\Sigma_{\infty}$.
\end{prop}

\begin{proof}
All graphs here have finite valences. Then we define a map which sends a vertex $v$ of $\Gamma_{\infty}^{*}$ (resp. $\Sigma_{\infty}^{*}$) to any vertex on the boundary of the face of $\Gamma_{\infty}$ (resp. $\Sigma_{\infty}$) corresponding to $v$. Then it is easy to check that this map is a quasi-isometry. For details, we refer to \cite[Theorem C]{merenkov1}.
\end{proof}

\begin{theorem}\label{equtype}
$\Gamma_{\infty}^{*}$ is quasi-isometric to $\Sigma_{\infty}^{*}$.
\end{theorem}

\begin{proof}
To prove the theorem, we need to construct a quasi-isometry $\varphi$ between $\Gamma_{\infty}^{*}$ and $\Sigma_{\infty}^{*}$, both of which are endowed with the word metric. As we noted before, the graph $\Gamma_{\infty}^{*}$ can be obtained by embedding a half-plane lattice in each face of the tree $T$; while $\Sigma_{\infty}^{*}$ is obtained by embedding a half-plane lattice in each face of the kernel $\k(T)$ of the tree $T$. Recall that $\k(T)$ is obtained from $T$ by cutting all finite trees attached to some vertices. 

The quasi-isometry will be constructed by gluing two maps: the first one is a quasi-isometry which maps the tree $T$ to its kernel $\k(T)$ and the second one is essentially a quotient map on half-plane lattices.

To define the required maps, we call every vertex $v$ on the kernel $\k(T)$ a \emph{parent}, and a vertex $v'$ on $T\setminus \k(T)$ is said to be a \emph{child} of $v$ if $v'$ can be connected to $v$ through a path without intersecting with the kernel. See Figure \ref{pa} as an illustration. Therefore, by definition every child has exactly one parent since the tree endowed with the word metric is a geodesic metric space where two vertices can be connected by a unique geodesic. Moreover, it follows from the topological uniformness condition that every parent has at most finitely many children.

\begin{figure}
    \begin{minipage}[t]{0.5\linewidth}
    \centering
    \includegraphics[width=7cm]{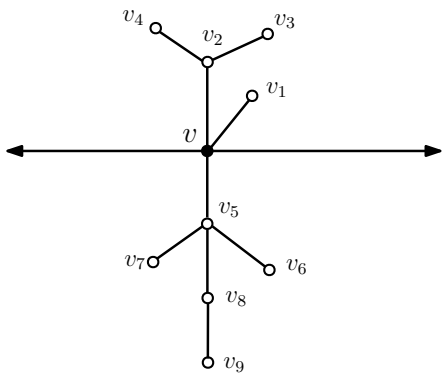}
    \end{minipage}%
 \begin{minipage}[t]{0.5\linewidth}\vspace{-6cm}
\caption{An illustration of the notion of parents and children. Shown is a portion of the tree $T$ around a vertex on the kernel in which the arrowed line denotes part of the kernel $\k(T)$. By definition, the vertex $v$ is a parent, while all the other vertices marked, namely $v_1, v_2, \dots, v_9$ in the figure, are the children of $v$.}
    \label{pa}
    \end{minipage}
\end{figure}

We will define a map
$$\varphi: \Gamma_{\infty}^{*}\to\Sigma_{\infty}^{*}$$
by
\begin{equation}\label{eq31}
\varphi(v)=
\begin{cases}
~\varphi_1(v), & \text{$v\in T$~;}\\
~\varphi_2(v), & \text{$v\in \Gamma_{\infty}^{*}\setminus T$}
\end{cases}
\end{equation}
such that $\varphi$ is a quasi-isometry. We start with the construction of the quasi-isometry 
$$\varphi_1: T \to \k(T).$$
The map is defined locally. To be more precise, with $v$ as a parent the map $\varphi_1$ sends the union of $v$ and all its children to $v$. Then the topological uniformness condition implies that $\varphi_1$ is actually a quasi-isometry. The first condition in Definition \ref{qi} is easy to check. For the second one we choose two arbitrary vertices $v_1$ and $v_2$ on $T$. By $(T3)$ in the TUC there is a constant $M$ not depending on $v_1$ and $v_2$ such that $\dist(v_j, \k(T))\leq M$ for $j=1,2$. Moreover, suppose that the parents of $v_j$ are $u_j$ for $j=1,2$. Then by our definition of $\varphi_1$ we see that $\varphi_1(v_j)=u_j$. Thus
\begin{equation}\label{est1}
d(\varphi_1(v_1), \varphi_1(v_2))=d(u_1,u_2)\leq d(v_1,v_2).
\end{equation}
On the other hand we know that
$$d(v_1,v_2)=d(u_1,u_2)+d(v_1,\k(T))+d(v_2,\k(T))\leq d(u_1,u_2) + 2M.$$
Therefore, we see that
\begin{equation}\label{est2}
d(\varphi_1(v_1), \varphi_1(v_2))=d(u_1,u_2)\geq d(v_1,v_2)-2M.
\end{equation}
Combining the above two estimates \eqref{est1} and \eqref{est2} we immediately see that $\varphi_1 : T\to\k(T)$ is a quasi-isometry.
\begin{figure}[htbp]
\centering
\includegraphics[width=12cm]{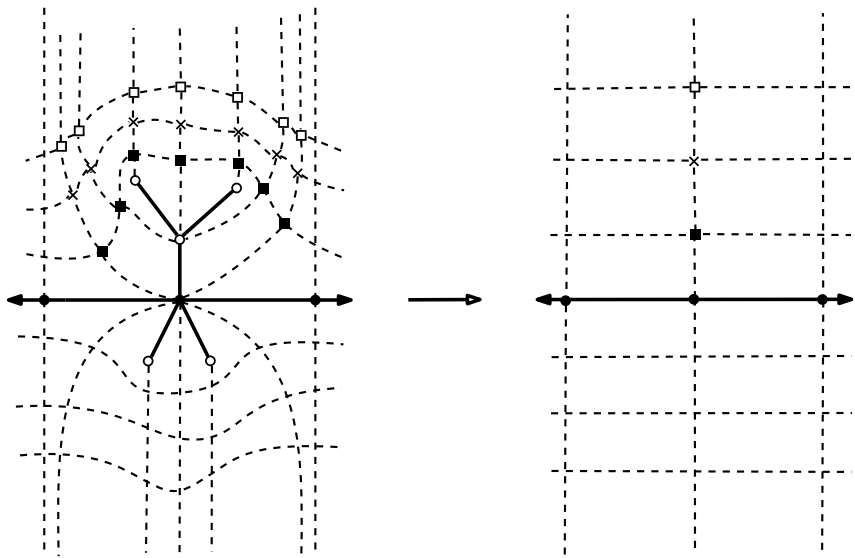}
\caption{The quasi-isometric map is defined by sending every vertex in $T$ to its parent, every vertex which is not in $T$ and has combinatorial distance $k$ to the unique vertex which is not in $\k(T)$ and has combinatorial distance $k$ to the kernel as shown in the figure. In the figure, all boxes, crosses and square vertices on the left are sent to the box, cross and square vertex on the right respectively. All vertices on the tree but not on the kernel (marked as circles) are sent to their parents.}
\label{quasimap}
\end{figure}

\medskip
To get the map $\varphi$ we still need to find a map $\varphi_2$ which is defined on the remaining part of $\Gamma_{\infty}^{*}$. For this purpose, let $u$ be a vertex on $\Gamma_{\infty}^{*}\setminus T$ (this is equivalent to saying that $u$ belongs to the embedded half-plane lattice). Then there is exactly one vertex on $T$ which minimizes the distance between $u$ and $T$. Thus as we defined before, this unique vertex has a unique parent $v$ on $\k(T)$. In this sense, we also say that $v$ is a parent of $u$. Denote by $T_v$ the union of $v$ and all its children on $T$. Moreover, suppose that the distance between $u$ and $T_v$ is $k:=k(u)\in\mathbb{N}$. Note that for the parent $v$ of $u$, in each component of $\C\setminus\k(T)$ there is exactly one vertex $u'$ in $\Sigma_{\infty}^{*}\setminus \k(T)$ such that $\dist(u',v)=\dist(u',\k(T))=k(u)$. Then we define the map 
$$\varphi_2:\Gamma_{\infty}^{*}\setminus T \to \Sigma_{\infty}^{*}\setminus \k(T), ~\quad \varphi_2(u)=u',$$
where $u'$ is the unique vertex corresponding to $u$ as discussed above. Now we show that $\varphi_2$ is a quasi-isometry in each component of $\Gamma_{\infty}^{*}\setminus T$. Notice that $\Gamma_{\infty}^{*}\setminus T$ is not connected but consists of finitely many components due to $(T2)$ in the TUC. We only need to check the second condition in Definition \ref{qi}. Let $u_1$ and $u_2$ be two vertices in one component of $\Gamma_{\infty}^{*}\setminus T$ whose parents are $v_1$ and $v_2$ respectively. First note that
\begin{equation}\label{est3}
d(\varphi_2(u_1), \varphi_2(u_2))\leq d(u_1,u_2),
\end{equation}
since basically the map $\varphi_2$ is a shrinking of the half-plane lattice. From the other side, without loss of generality we assume that $k(u_1)\geq k(u_2)$. Since by the TUC, at each $v_j$ there are only finitely many vertices on each component of $T\setminus\k(T)$ and the number is uniformly bounded above by some constant $A$ depending only on the TUC.
Now we have
\begin{equation}\label{est4}
\begin{aligned}
d(u_1, u_2)&\leq d(v_1,v_2)+2A\cdot B + k(u_1)-k(u_2)\\
&=d(v_1, v_2)+2A\cdot B+d\left(\varphi_2(u_1),v_1\right)-d\left(\varphi_2(u_2),v_2\right)\\
&\leq d(\varphi_2(u_1), \varphi_2(u_2))+2A\cdot B,
\end{aligned}
\end{equation}
where $B$ is a constant which bounds the local valence of $T$. By $(T1)$ this constant depends only on the tree. Therefore, $\varphi_2$ is a quasi-isometry on each component by \eqref{est3} and \eqref{est4}.

\bigskip
Now it is left to show that the map $\varphi$ defined in \eqref{eq31} is a quasi-isometry. Clearly it is only necessary to check the case when two vertices are chosen in different components of $\Gamma_{\infty}^{*}\setminus T$. Let $u_1, u_2$ be two such vertices with parents $v_1$ and $v_2$ respectively. Then
\begin{equation}\label{est5}
\begin{aligned}
d(u_1, u_2)&\leq d(u_1, T_{v_1})+M+d(v_1, v_2)+d(u_2, T_{v_2})+M\\
&=d(\varphi(u_1), v_1)+d(v_1, v_2)+d(\varphi(u_2), v_2)+2M\\
&=d(\varphi(u_1),\varphi(u_2))+2M.
\end{aligned}
\end{equation}
Moreover, it is easy to see that
\begin{equation}\label{est6}
d(\varphi(u_1), \varphi(u_2))\leq d(u_1,u_2).
\end{equation}
Thus, by the above discussion the map $\varphi$ is a quasi-isometry. For an illustration of the map $\varphi$, see Figure \ref{quasimap}.
\end{proof}

To prove Theorem \ref{rea} we will show that $\Sigma_{\infty}^{*}$ is of parabolic type. To this aim, we introduce a standard model $\Sigma_s$ for each $\Sigma_{\infty}^{*}$. Let $N$ be the number of complementary components of $\k(T)$ in the plane. Denote by $T_s$ an infinite tree in the plane which satisfies:
\begin{itemize}
\item[(1)] $T_s$ has $N$ complementary components in the plane;
\item[(2)] $T_s$ has exactly one vertex of valence $N$ and all other vertices have local valence two.
\end{itemize}
Then the standard model $\Sigma_{s}$ is defined from $T_s$ by embedding in each complementary components of $T_s$ a half-plane lattice, identifying the boundary vertices of $T_s$ with the boundary vertices of the half-plane lattice. See the right-hand side of Figure \ref{standard22} in case that $N=4$ where $T_s$ is the graph marked as solid lines.

\begin{figure}[htbp]
    \centering
     \includegraphics[width=13cm]{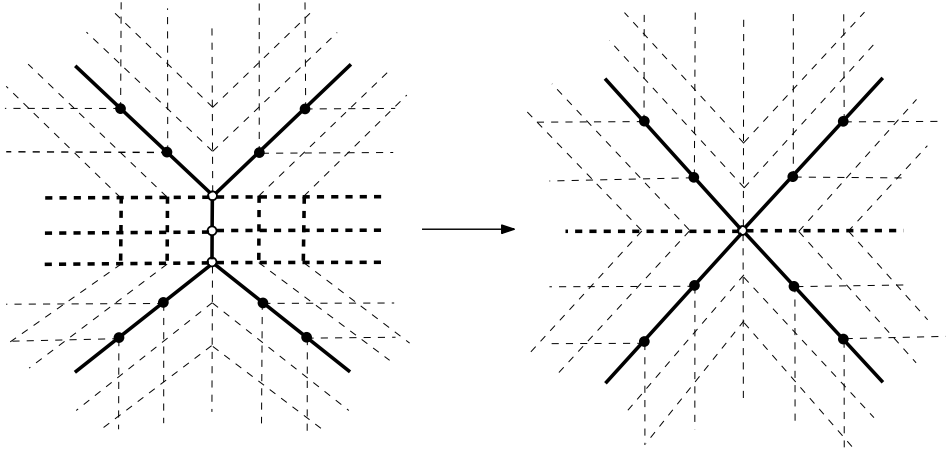}
     \caption{On the left is the kernel $\k(T)$ of a tree $T$ (black) and $\Sigma_{\infty}^{*}$. On the right hand side is the standard model $\Sigma_s$. The map is then defined as identifying lines marked as deep dark dashed lines.}
     \label{standard22}
\end{figure}

\begin{prop}\label{prop2}
$\Sigma_{\infty}^{*}$ is quasi-isometric to $\Sigma_s$.
\end{prop}

\begin{proof}
We only sketch the idea of proof. First define a quasi-isometry between $\k(T)$ and $T_s$, which is, roughly speaking, sending all vertices with local valence greater than two and all vertices in between to the only vertex in $T_s$ with local valence greater than two. Then one needs to define a map from $\Sigma_{\infty}^{*}\setminus\k(T)$ to $\Sigma_s\setminus T_s$, which is similar as that of the map $\varphi_2$ defined in the proof of Theorem \ref{equtype}. See Figure \ref{standard22} for an explanation. The details are left to the reader.
\end{proof}

\begin{theorem}\label{pker}
$\Sigma_{s}$ is parabolic.
\end{theorem}

\begin{proof}
We sketch a proof. Suppose that $\k(T)$ has $N$ complementary components in the plane. According to Proposition \ref{howto}, it is enough to find a sequence of disjoint annuli $\{A_n\}$ in $\Sigma_s$ such that
\begin{equation}\label{pm}
\sum \frac{1}{\mod A_n}=\infty.
\end{equation}
To this aim, let $A_n$ be a finite graph which has combinatorial width $1$ in each component of $\Sigma_s \setminus T_s$. See Figure \ref{annuli}.  We consider a mass distribution on $\Sigma_s$ which assigns mass one to every edge. It is easy to check that this mass distribution is admissible for every chain. 
Then we can see that
$$\mod A_n =\mathcal{O}(n).$$
which implies \eqref{pm}. This completes the proof of the theorem.
\end{proof}

\begin{figure}
    \begin{minipage}[t]{0.6\linewidth}
    \centering
    \includegraphics[width=8cm]{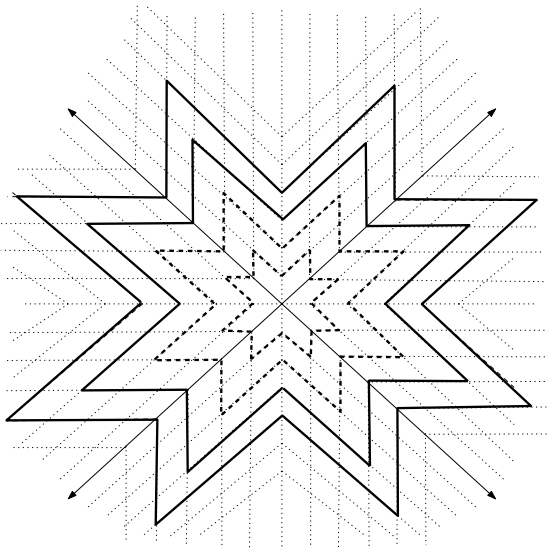}
    \end{minipage}%
 \begin{minipage}[t]{0.4\linewidth}\vspace{-7cm}
\caption{Shown in the figure are two annuli. The first annulus is $A_1$ with inner and outer boundaries marked as dashed lines, while the second annulus is $A_2$ with boundaries marked as solid lines. Both annuli have combinatorial width $1$.}
    \label{annuli}
    \end{minipage}
\end{figure}

\begin{proof}[Proof of Theorem \ref{rea}]
In some sense, our theorem says that the type of the surface is determined by the kernel of the tree. It follows from the Doyle-Merenkov criterion that the surface $(X,p)$ corresponding to $\Gamma$ is parabolic if and only if $\Gamma_{\infty}$ is parabolic. By Theorem \ref{stable} and Proposition \ref{equ1} we see that $\Gamma_{\infty}$ is parabolic if and only if $\Gamma_{\infty}^{*}$ is parabolic. Proposition \ref{prop2} and again Theorem \ref{stable} say that this holds if and only if $\Gamma_{s}$ is parabolic. Theorem \ref{pker} says that this is indeed the case. Thus $\Gamma_{\infty}$ is parabolic and hence the surface $(X,p)$ is parabolic. Therefore, we obtain an entire function $f$ with two critical values $\pm 1$ and no asymptotic values such that $T_f$ is homeomorphic to the given tree $T$.
\end{proof}

\section{Nevanlinna's theorem revisited}

Using similar ideas as in the proof of Theorem \ref{rea}, we give here a topological and combinatorial proof of a result due to Nevanlinna \cite{nevanlinna1} and Elfving \cite{elfving1}. For this purpose we need the notion of a logarithmic singularity. We follow the definition in \cite{bergweiler18}.

Let $(X,p)$ be a surface spread over the sphere in class $\s$. Denote by $A$ the set of singular values of $p$. Then
$$p: X\setminus p^{-1}(A) \to \hc\setminus A$$
is a covering map. Let $b\in A$ be fixed and $D_{\chi}(b,r)$ be a spherical disk centered at $b$ with radius $r$. Consider a function $U: r\mapsto U(r)$ which to each $r>0$ assigns a component $U(r)$ of $p^{-1}(D_{\chi}(b,r))$ in such a way that $r_1 < r_2$ implies that $U(r_1)\subset U(r_2)$. If $\cap_{r>0}U(r)=\emptyset$, then we say that $r\mapsto U(r)$ defines a singular point over $b$. A singular point over $b$ is said to be \emph{logarithmic} if for some $r>0$ the map $p|_{U(r)}: U(r)\to D_{\chi}(b,r)\setminus\{b\}$ is a universal covering.

Nevanlinna's theorem is then stated as follows.

\begin{theorem}[Nevanlinna, Elfving]\label{n}
Let $(X,p)$ be a surface spread over the sphere in the Speiser class with finitely many logarithmic singularities $(\geq 2)$ and finitely many critical points. Then the surface is parabolic.
\end{theorem}

The above form is due to Elfving \cite{elfving1}. Nevanlinna proved his theorem without allowing critical points \cite{nevanlinna1}. His proof uses Speiser graphs and some theory from complex differential equations; see also \cite[Chapter XI]{nevanlinna}. Right after the paper \cite{nevanlinna1}, Ahlfors gave another proof, in which the connection with quasiconformal mappings is implicitly mentioned \cite{ahlfors19}. Recently, a purely analytic approach to the above theorem is given by Langley \cite{langley1} using the Wiman-Valiron property of meromorphic functions with logarithmic singularities \cite{bergweiler3}. Here, we supply a purely combinatorial proof, which also uses Speiser graphs, but the theory of complex differential equations is not used. 

Since surfaces satisfying the condition in the theorem are parabolic, there exist meromorphic functions in the plane corresponding to these surfaces. The importance of these surfaces lies in the fact that the corresponding meromorphic functions are extremal in the sense of Nevanlinna's inverse problem. See \cite[Chapter 7]{goldbergmero} for explanations and more details, and also a proof of the above theorem.

\begin{proof}[Proof of Theorem $\ref{n}$]
The proof will be similar as that of Theorem \ref{rea}. The key observation is that up to a large disk in the graph which contains all critical points, the Speiser graph is essentially a ``tree" (when ``observed from far away").

We assume that the surface $(X,p)$ has $q$ logarithmic singularities and $m$ critical points. See Figure \ref{nevan34} for some examples with and without critical points.
\begin{figure}
    \begin{minipage}[t]{0.5\linewidth}
    \centering
    \includegraphics[width=5cm]{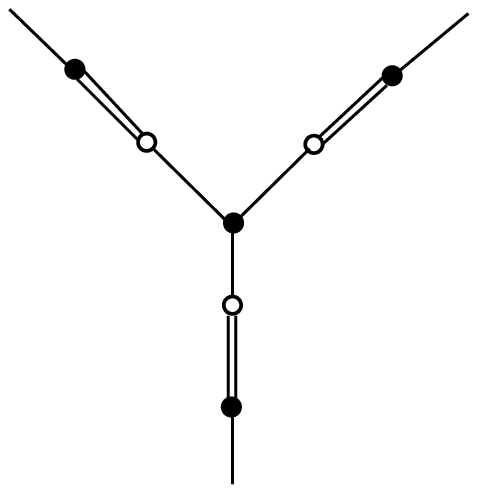}
    \end{minipage}%
    \begin{minipage}[t]{0.5\linewidth}
    \centering
    \includegraphics[width=5cm]{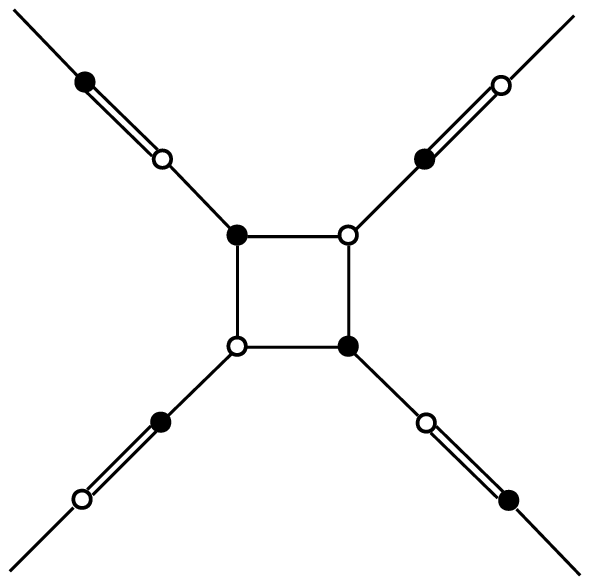}
    \end{minipage}%
    \caption{Some Speiser graphs. On the left is shown a Speiser graph with three logarithmic singularities and without critical points; the corresponding meromorphic function has polynomial Schwarzian derivative. On the right-hand side is a Speiser graph with one critical point (corresponding to the face with four edges on the boundary) and with four logarithmic singularities; the function is actually $z\mapsto e^{z^2}$. Both surfaces are ramified over three values.}
    \label{nevan34}
\end{figure}

We denote by $\Gamma$ the corresponding Speiser graph corresponding to $(X,p)$. By Theorem \ref{dmc}, to show that $(X,p)$ is parabolic we only need to embed a half-plane lattice in each unbounded face of $\Gamma$ to obtain an extended Speiser graph $\Gamma_{\infty}$ and show that $\Gamma_{\infty}$ is parabolic (note that we can do so since there is a uniform bound on the number of edges of faces). To achieve this, we need to do some surgery on the Speiser graph $\Gamma$. We first replace each face with finitely many edges on the boundary by a single vertex. Then if two vertices are connected by more than one edges, we replace all these edges by a single edge. After these operations, what is left is actually a tree which coincides with its kernel, denoted by $T$. Now in each unbounded face of $T$ we, as before, embed a half-plane lattice. Then we get a graph, denoted by $\Sigma_{\infty}$. It follows from the proof of Theorem \ref{rea} that $\Sigma_{\infty}$ is parabolic.

The rest of the proof proceeds as that of Theorem \ref{rea}. Since there are only finitely many critical points, it is easy to establish a quasi-isometry between $\Gamma$ and $T$. Then we establish a quasi-isometry in each component of $\Gamma_{\infty}\setminus \Gamma$ to the corresponding component of $\Sigma_{\infty}\setminus T$. This is similar as before. Then combining all these quasi-isometries we can define a quasi--isometry between $\Gamma_{\infty}$ and $\Sigma_{\infty}$. The parabolicity of $\Sigma_{\infty}$ implies that $\Gamma_{\infty}$ is parabolic. Thus $(X,p)$ is parabolic. The theorem follows.
\end{proof}

\section{Examples}

This section is devoted to the construction of some examples concerning the items in the TUC. On the one hand, we construct Shabat entire functions which show that every item in the TUC can be slightly extended. On the other hand, we show that if any one of the items is dropped, there may not be Shabat entire functions.

\subsection{Growing local valence}

We consider trees with growing local valences. First we show that there exist Shabat entire functions such that the local valence of preimages of one critical value tends to infinity.

\subsubsection{Parabolicity} There exist Shabat entire functions with unbounded local valence. Such functions can actually be constructed using the so-called MacLane-Vinberg method. Moreover, entire functions constructed in this way belong to the \emph{Laguerre-P\'olya class}, denoted by $\mathcal{LP}$, which consists of entire functions that are locally uniform limits of real polynomials with only real zeros. Thus such entire functions have only real zeros.

Let us consider the example in \cite[Example 1.7]{bergweiler6}, which was constructed using the above MacLane-Vinberg method and hence belongs to the Laguerre-P\'olya class. More precisely, the function, denoted by $g$, has two critical values $0$ and $1$ and no finite asymptotic values. It is easy to see that the map $f(z):=2g(z)-1$ is an entire function with two critical values $\pm 1$ and no finite asymptotic values, and hence is a Shabat entire function. By the properties of $\mathcal{LP}$ functions, the preimages of $-1$ are all real (which correspond to all zeros of $g$) whose multiplicities tend to $\infty$ quite fast. See Figure \ref{bfgtree}.

\begin{figure}[htbp] 
    \centering
     \includegraphics[width=13cm]{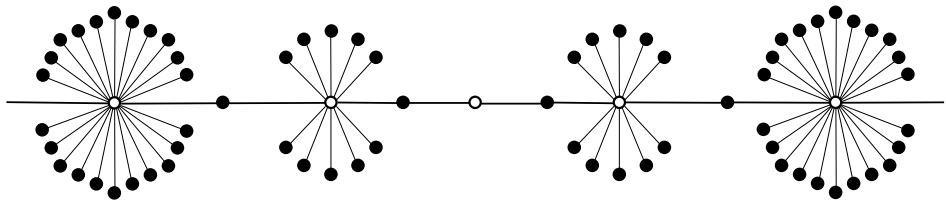}
     \caption{A rough sketch of the tree corresponding to the functions constructed in \cite{bergweiler6}. The local valence of critical points (marked as circles) corresponding to the critical value $-1$ tends to infinity fast.}
     \label{bfgtree}
\end{figure}

We provide here another example. For this and also for later purposes, we start with the Speiser graph of the sine function, which is shown in Figure \ref{sinegraph}. 
\begin{figure}[htbp] 
    \centering
     \includegraphics[width=13cm]{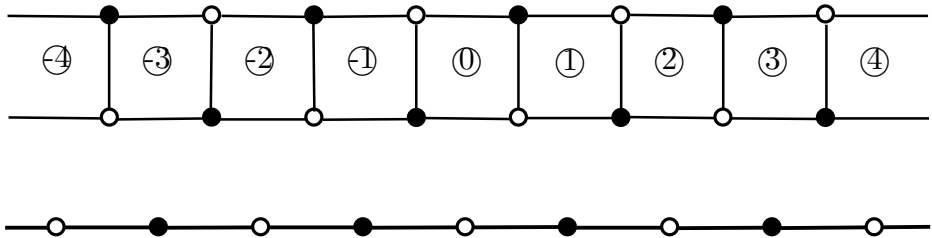}
     \caption{The Speiser graph of the sine function and the associated tree. In the graph, we choose one face as a starting face marked as \textcircled{0} for our constructions afterwards. The face to the right and the left are marked as \textcircled{1} and \textcircled{-1} respectively. We continue this marking as shown in the figure.}
     \label{sinegraph}
\end{figure}
Now on the upper and lower edge of face \circled{2} we add two vertices respectively such that the type of the vertices are arranged to be compatible such that the newly produced graph is still a Speiser graph (actually this can be done as along as we add an even number of vertices on the edge). We do a similar surgery to the face \circled{-2}. Then we continue this to all faces marked as \circled{2k}, where $2|k|$ vertices are added on the upper and lower edges of face \circled{2k} respectively and then add certain additional edges to make the graph to be a Speiser graph (local valence is $3$). Denote this Speiser graph by $\Gamma$. We obtain the following graph, shown in Figure \ref{unbg}. 
\begin{figure}[htbp] 
    \centering
     \includegraphics[width=16cm]{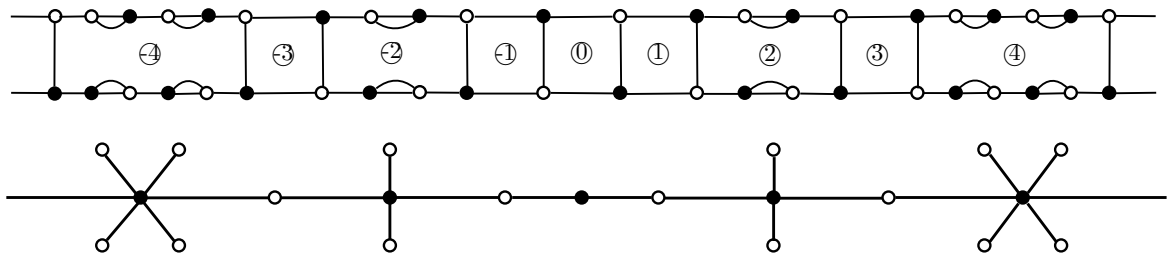}
     \caption{A Speiser graph with growing local valence and its corresponding tree.}
     \label{unbg}
\end{figure}
Now we need to show that $\Gamma$ corresponds to a surface spread over the sphere of parabolic type. Note first that associated with $\Gamma$ is a tree which is locally finite, satisfies the condition $(T2)$ and $(T3)$ in the TUC but with growing local valence (that is, $(T1)$ is violated). We use a criterion due to Nevanlinna and Wittich \cite[Chapter VII]{wittich1}, which can also be used to prove Nevanlinna's theorem in case that there are no critical points. Since every vertex in the graph can be connected to $\infty$ with a path in $\C$ without intersecting with $\Gamma$, to use the abovementioned criterion we only need to fix a vertex as generation $0$. Then the generation $1$ will be the vertices that can be connected to the generation $0$. Let $\Gamma^n$ be the portion of $\Gamma$ with vertices of generation at most $n$. Then a vertex in $\Gamma^n$ is in $\partial\Gamma^n$ if it can be connected to infinity with a path in $\C$ with no intersection with $\Gamma$. Denote by $s(n)$ the number of vertices on $\partial\Gamma^n$. Nevanlinna-Wittich's criterion states as follows: The surface spread over the sphere corresponding to $\Gamma$ is parabolic if $\sum \frac{1}{s(n)}=\infty$. It is easy to check that the graph in Figure \ref{unbg} satisfies that $s(n)=\mathcal{O}(n)$ and hence the existence of a Shabat entire function with a growing local valence follows.

\begin{remark}
Another method to show that the surface associated with the Speiser graph in Figure \ref{unbg} is parabolic is by using Theorem \ref{dmc}: We consider the extended Speiser graph $\Gamma_4$ and then find a sequence of disjoint annuli of width one to show that Proposition \ref{howto} is satisfied. Details are omitted.
\end{remark}

\subsubsection{Hyperbolicity}

In this part, we will construct a tree which satisfies the TUC except $(T1)$. Our example here is based on the surface given by Sutter \cite{sutter}; see also \cite{pfluger1}. The Speiser graph of this surface, denoted by $\Sigma$, is shown in Figure \ref{sutterg}. 
\begin{figure}[htbp] 
    \centering
     \includegraphics[width=11cm]{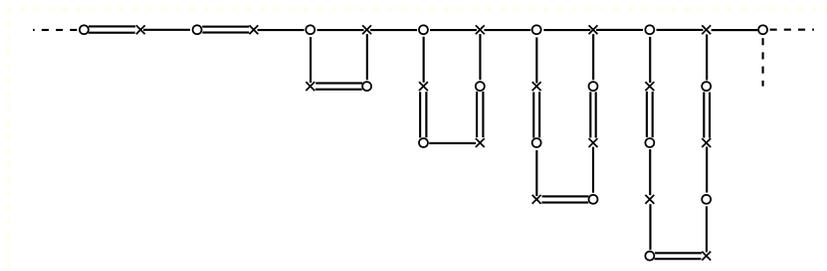}
     \caption{The Speiser graph used by Sutter to construct a hyperbolic surface.}
     \label{sutterg}
\end{figure}
However, this surface does not meet our requirement since it has a logarithmic singularity over some finite value. Nevertheless, we can modify this example for our aim. More precisely, we consider the Speiser graph $\Gamma$ shown in Figure \ref{og}, where two unbounded faces represent two logarithmic singularities over infinity while faces with finitely many edges on the boundary correspond to critical points which are preimages of critical values $\pm 1$. The corresponding tree is also shown in the figure.
\begin{figure}[htbp]
    \centering
     \includegraphics[width=11cm]{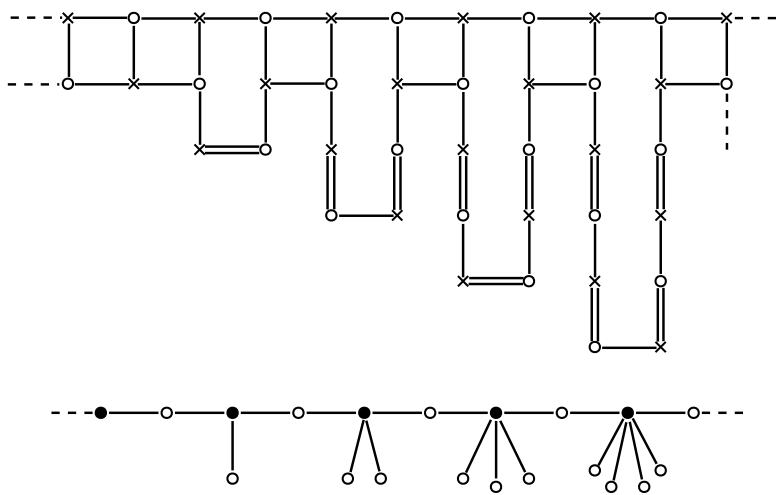}
     \caption{The Speiser graph and the corresponding tree that we use to construct a hyperbolic surface. There exists a sequence of critical points (marked as disks in the tree) with multiplicities tending to infinity.}
     \label{og}
\end{figure}

We need to show that the surface corresponding to our Speiser graph is hyperbolic. For this aim, we consider the extended Speiser graphs $\Sigma_3$ and $\Gamma_3$ of $\Sigma$ and $\Gamma$ respectively. By Theorem \ref{dmc} and the hyperbolicity of the surface corresponding to $\Sigma$ we know that $\Sigma_3$ is hyperbolic. Moreover, to show that the surface corresponding to $\Gamma$ is hyperbolic, again by Theorem \ref{dmc} it is enough to show that $\Gamma_3$ is hyperbolic. Now by Theorem \ref{stable} we only need to show that $\Gamma_3$ is quasi-isometric to $\Sigma_3$. But this is easy to establish and we omit the detail. Moreover, we can arrange the markings of the Speiser graph $\Gamma$ such that the two unbounded faces represent two logarithmic singularities over $\infty$ while all finite faces correspond to critical values $\pm 1$.

\subsection{Non-uniform distance from the kernel}

\subsubsection{Parabolicity}

To show that there exists a tree which has a true form, we consider the example in \cite[Section 3.3]{merenkov3}. The example constructed there has exactly two critical values $0$ and $1$ and no finite asymptotic values. Moreover, there is a sequence of vertices on the tree whose distance to the kernel tends to infinity. The surface spread over the sphere corresponding to this tree is parabolic, and thus there exists a Shabat entire function (by using affine maps which maps $0$ and $1$ to $-1$ and $1$ respectively). The vertices of this tree do not have a uniform distance to the kernel.
\begin{figure}[htbp]
    \centering
     \includegraphics[width=16cm]{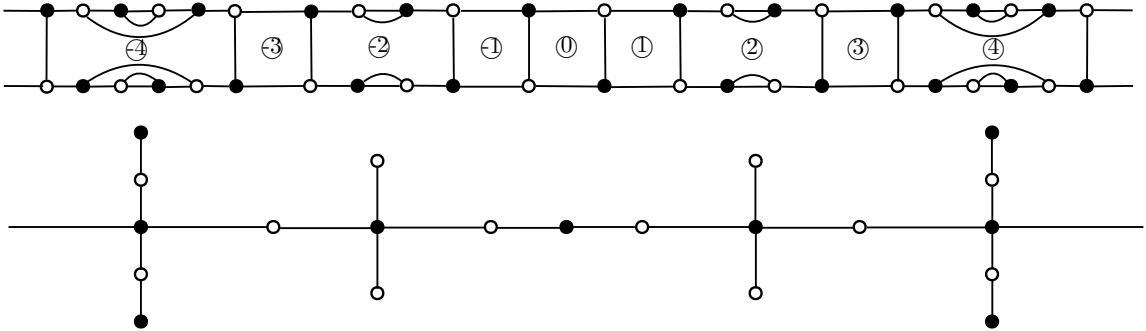}
     \caption{A Speiser graph (the upper figure) associated to a tree (the lower figure) with unbounded distance to the kernel.}
     \label{und}
\end{figure}

Another example can be constructed as follows. We start again with the Speiser graph of the sine function, as shown in Figure \ref{sinegraph}. The marking stays the same. As before, on the upper and lower edges of the face \circled{2k} we add $2|k|$ vertices respectively. But this time we obtain a different Speiser graph by connecting these newly added vertices in another way, which is shown in Figure \ref{und}. Denote this graph by $\Gamma$.

The parabolicity of the surface corresponding to this Speiser graph follows from the Nevanlinna-Wittich criterion as we used before. As a matter of fact, we can also use Theorem \ref{dmc}. This is due to the fact that every face in the graph has a uniformly bounded number of edges on the boundary ($\leq 8$) except for those faces with infinitely many edges on the boundary. Then one can consider the extended Speiser graph $\Gamma_{5}$, which is actually obtained by embedding in each unbounded face a half-plane lattice. To show the parabolicity of $\Gamma_5$ it is necessary to find a sequence of disjoint nested annuli and use Proposition \ref{howto}. We leave the details to the reader.

\subsubsection{Hyperbolicity}

As before, we start with the Speiser graph of the sine function, see Figure \ref{sinegraph}. On the upper edges of faces \circled{1} and \circled{-1} we add $A_1$ vertices respectively. On the upper edges of faces \circled{2} and \circled{-2} we add $A_2$ vertices respectively. More generally, on the upper edge of face \circled{k} we add $A_k$ vertices. Here $A_k\in\mathbb{N}$ and $A_k \to\infty$ as $k\to\infty$. Connecting these newly added vertices properly, we can get a Speiser graph of valence $3$. Denote by $\Gamma$ the Speiser graph obtained. In Figure \ref{undhy} it is shown a Speiser graph constructed from the above surgery.

\begin{figure}[htbp]
    \centering
     \includegraphics[width=16cm]{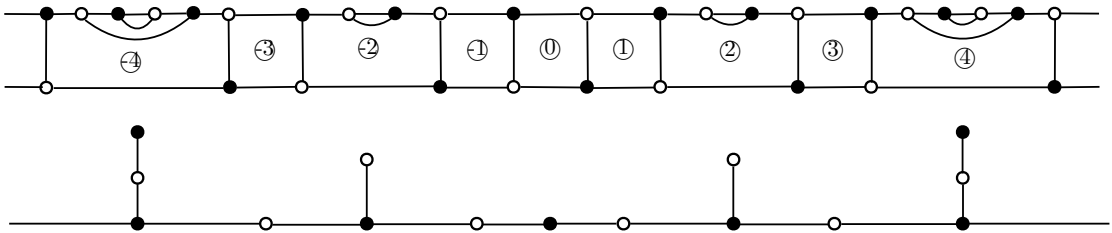}
     \caption{A Speiser graph associated to a tree with non-uniform distance to the kernel. The tree below is just half of the one in Figure \ref{und}.  The extended Speiser graph will be hyperbolic.}
     \label{undhy}
\end{figure}

The Speiser graph $\Gamma$ has two faces with infinitely many edges on the boundaries and all remaining faces have a uniformly bounded number of edges on the boundaries ($\leq 6$). Therefore, we can consider the extended Speiser graph $\Gamma_{4}$, which is actually obtained by embedding one half-plane lattice in the two unbounded faces. Now if $\Gamma_4$ is hyperbolic, then by Theorem \ref{dmc} there does not exist any entire function.

To show that $\Gamma_4$ is hyperbolic, we shall use an isoperimetric inequality. Let $V$ be a vertex set in a graph $G$ and $\partial V$ the boundary of $V$ consisting vertices with neighbours outside of $V$. Let $f$ be a non-decreasing positive real function defined on $\mathbb{N}$. We say that $G$ satisfies an $f$-isoperimetric inequality if there exist a constant $c>0$ such that, for each finite vertex set $V$ of $G$,
\begin{equation}\label{isop}
|\partial V| > c f(|V|),
\end{equation}
where $|\cdot|$ denote the cardinality. If the above inequality holds for all finite vertex sets $V$ which contains a fixed vertex (root) $v$ and induce connected subgraphs in $G$, then $G$ is said to satisfy a rooted, connected $f$-isoperimetric inequality. A theorem by Thomassen says that each connected graph with a uniformly bounded local valence satisfying a rooted, connected $f$-isoperimetric inequality for $f(k)=k^{1/2+\varepsilon}$ and some $\varepsilon>0$ is hyperbolic\footnote{Here $\varepsilon$ cannot be equal to $0$, since the two dimensional lattice is parabolic by a theorem of P\'olya and in this case $\varepsilon=0$.}. See \cite[Theorem 3.4]{thomassen}. It is easy to see that, for the graph $\Gamma_4$ if we let $A_k$ tends to $\infty$ rapidly fast, then the inequality \eqref{isop} is satisfied for some $\varepsilon>0$. Hence $\Gamma_4$ is hyperbolic.

\subsection{Infinitely many complementary components}

As we discussed before, any homogeneous tree of valence no less than $3$ does not give us any entire function. On the contrary, to show that there exists a Shabat entire function from a tree with infinitely many complementary components, one can use Bishop's technique of quasiconformal folding. As a simple example, we can consider the tree in \cite[Figure 5]{bishop1}, with which Bishop constructed a Shabat entire function with rapid growth. In fact, it follows from the Denjoy-Carleman-Ahlfors theorem (cf. \cite[Chapter 5, Theorem 1.4]{goldbergmero}) that any Shabat entire function constructed from trees with infinitely many complementary components must be of infinite order.

We will consider the tree obtained from the one in \cite[Figure 5]{bishop1} after the application of the quasiconformal folding. It follows from this technique that this tree will satisfy our condition $(T1)$ and $(T2)$ but not $(T3)$. The details are omitted.

Another more direct example of trees with infinitely many complementary components but satisfying $(T1)$ and $(T2)$ is obtained by considering the function $f(z)=\cos(\pi\cos(z))$. A sketch of the corresponding tree is shown in Figure \ref{cosinetree}.
\begin{figure}[htbp]
    \centering
     \includegraphics[width=10cm]{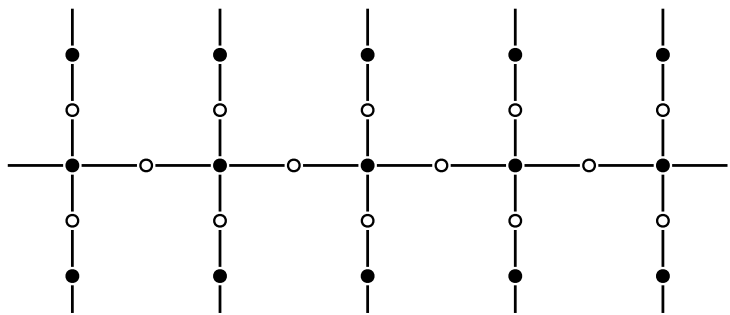}
     \caption{A sketch of the tree $T_f$ for $f(z)=\cos(\pi\cos(z))$. Here dots are the preimages of $-1$ and circles are the preimages of $+1$.}
     \label{cosinetree}
\end{figure}


\begin{thebibliography}{DGPC05}

\bibitem[Ahl32]{ahlfors19}
L.~V. Ahlfors, \emph{{\"Uber eine in der neueren Wertverteilungstheorie
  betrachtete Klasse transzendenter Funktionen}}, Acta Math. \textbf{58}
  (1932), no.~1, 375--406.

\bibitem[Ahl10]{ahlfors9}
L.~V. Ahlfors, \emph{Conformal invariants: Topics in geometric function theory,
  {R}eprint of the 1973 original, {W}ith a foreword by {P}eter {D}uren, {F}.
  {W}. {G}ehring and {B}rad {O}sgood}, AMS Chelsea Publishing, Providence, RI,
  2010.

\bibitem[BE95]{bergweiler18}
W.~Bergweiler and A.~Eremenko, \emph{On the singularities of the inverse to a
  meromorphic function of finite order}, Rev. Mat. Iberoam. \textbf{11} (1995),
  no.~2, 355--373.

\bibitem[BFRG15]{bergweiler6}
W.~Bergweiler, N.~Fagella, and L.~Rempe-Gillen, \emph{Hyperbolic entire
  functions with bounded {Fatou} components}, Comment. Math. Helv. \textbf{90}
  (2015), no.~4, 799--829.

\bibitem[Bis14]{bishop4}
C.~J. Bishop, \emph{True trees are dense}, Invent. Math. \textbf{197} (2014),
  no.~2, 433--452.

\bibitem[Bis15]{bishop1}
C.~J. Bishop, \emph{Constructing entire functions by quasiconformal folding}, Acta
  Math. \textbf{214} (2015), no.~1, 1--60.

\bibitem[BMS04]{merenkov2}
I.~Benjamini, S.~Merenkov, and O.~Schramm, \emph{A negative answer to
  {Nevanlinna's} type question and a parabolic surface with a lot of negative
  curvature}, Proc. Amer. Math. Soc. \textbf{132} (2004), no.~3, 641--647.

\bibitem[BRS08]{bergweiler3}
W.~Bergweiler, P.~J. Rippon, and G.~M. Stallard, \emph{Dynamics of meromorphic
  functions with direct or logarithmic singularities}, Proc. London Math. Soc.
  (3) \textbf{97} (2008), no.~2, 368--400.

\bibitem[Can94]{cannon1}
J.~W. Cannon, \emph{The combinatorial {R}iemann mapping theorem}, Acta Math.
  \textbf{173} (1994), no.~2, 155--234.

\bibitem[DGPC05]{drasin1}
D.~Drasin, A.~A. Gol'dberg, and P.~Poggi-Corradini, \emph{Quasiconformal
  mappings in value-distribution theory}, Handbook of Complex Analysis:
  Geometric Function Theory, {V}ol. 2, Elsevier Sci. B. V., Amsterdam, 2005,
  pp.~755--808.

\bibitem[Die17]{diestel}
R.~Diestel, \emph{Graph theory}, {F}ifth ed., Graduate Texts in Mathematics,
  vol. 173, Springer, Berlin, 2017.

\bibitem[Doy84]{doyle1}
P.~G. Doyle, \emph{Random walk on the {Speiser} graph of a {Riemann} surface},
  Bull. Amer. Math. Soc. (N.S.) \textbf{11} (1984), no.~2, 371--377.

\bibitem[Duf62]{duffin1}
R.~J. Duffin, \emph{The extremal length of a network}, J. Math. Anal. Appl.
  \textbf{5} (1962), 200--215.

\bibitem[EL92]{eremenko2}
A.~Eremenko and M.~Lyubich, \emph{Dynamical properties of some classes of
  entire functions}, Ann. Inst. Fourier \textbf{42} (1992), no.~4, 989--1020.

\bibitem[Elf34]{elfving1}
G.~Elfving, \emph{{\"Uber} eine {Klasse} von {Riemannschen Fl\"achen} und ihre
  {Uniformiserung}}, Acta Soc. Sci. Fenn. (N. S.) \textbf{2} (1934), no.~3,
  1--60.

\bibitem[Ere04]{eremenko5}
A.~Eremenko, \emph{Geometric theory of meromorphic functions}, In the tradition
  of {A}hlfors and {B}ers, {III}, Contemp. Math., vol. 355, Amer. Math. Soc.,
  Providence, RI, 2004, pp.~221--230.

\bibitem[ERG15]{epstein5}
A.~L. Epstein and L.~Rempe-Gillen, \emph{On invariance of order and the area
  property for finite-type entire functions}, Ann. Acad. Sci. Fenn. Math.
  \textbf{40} (2015), no.~2, 573--599.

\bibitem[GGD12]{girondo}
E.~Girondo and G.~Gonz\'alez-Diez, \emph{Introduction to compact {R}iemann
  surfaces and dessins d'enfants}, London Mathematical Society Student Texts,
  vol.~79, Cambridge University Press, Cambridge, 2012.

\bibitem[GO08]{goldbergmero}
A.A. Goldberg and I.V. Ostrovskii, \emph{Value distribution of meromorphic
  functions}, Translations of Mathematical Monographs 236, American
  Mathematical Society, 2008.

\bibitem[Gro81]{gromov1}
M.~Gromov, \emph{Hyperbolic manifolds, groups and actions}, Riemann surfaces
  and related topics: {P}roceedings of the 1978 {S}tony {B}rook {C}onference
  ({S}tate {U}niv. {N}ew {Y}ork, {S}tony {B}rook, {N}.{Y}., 1978), Ann. of
  Math. Stud., vol.~97, Princeton Univ. Press, Princeton, N.J., 1981,
  pp.~183--213.

\bibitem[Kan85]{kanai2}
M.~Kanai, \emph{{Rough isometries, and combinatorial approximations of
  geometries of non-compact Riemannian manifolds}}, J. Math. Soc. Japan
  \textbf{37} (1985), no.~3, 391--413.

\bibitem[Kan86]{kanai1}
M.~Kanai, \emph{Rough isometries and the parabolicity of {R}iemannian
  manifolds}, J. Math. Soc. Japan \textbf{38} (1986), no.~2, 227--238.

\bibitem[Lan16]{langley1}
J.~K. Langley, \emph{The {Schwarizian} derivative and the {Wiman-Valiron}
  property}, J. Anal. Math. \textbf{130} (2016), 71--89.

\bibitem[Mer03]{merenkov3}
S.~Merenkov, \emph{Determining biholomorphic types of a manifold using
  combinatorial and algebraic structures}, Ph.D. thesis, Purdue University,
  2003.

\bibitem[Mer08]{merenkov1}
S.~Merenkov, \emph{Rapidly growing entire functions with three singular values},
  Illinois J. Math. \textbf{52} (2008), no.~2, 473--491.

\bibitem[Nev32]{nevanlinna1}
R.~Nevanlinna, \emph{{\"U}ber {Riemannsche Fl\"achen} mit endlich vielen
  {Windungspunkten}}, Acta Math. \textbf{58} (1932), no.~1, 295--373.

\bibitem[Nev70]{nevanlinna}
R.~Nevanlinna, \emph{Analytic functions}, Translated from the second German edition
  by Phillip Emig. Die Grundlehren der mathematischen Wissenschaften, Band 162,
  Springer-Verlag, New York-Berlin, 1970.

\bibitem[PS66]{pfluger1}
A.~Pfluger and J.~Sutter, \emph{Riemannsche {F}l\"achen vom hyperbolischen
  {T}ypus, erzeugt durch {A}symmetrien}, Contemporary {P}roblems in {T}heory
  {A}nal. {F}unctions ({I}nternat. {C}onf., {E}revan, 1965) ({R}ussian), Izdat.
  ``Nauka'', Moscow, 1966, pp.~253--257.

\bibitem[Soa94]{soardi}
P.~M. Soardi, \emph{Potential theory on infinite networks}, Lecture Notes in
  Mathematics, no. 1590, Springer-Verlag, Berlin, 1994.

\bibitem[Sto56]{stoilow1}
S.~Sto{\"\i}low, \emph{Le\c{c}ons sur les principes topologiques de la
  th\'eorie des fonctions analytiques. {D}euxi\`eme \'edition, augment\'ee de
  notes sur les fonctions analytiques et leurs surfaces de {R}iemann},
  Gauthier-Villars, Paris, 1956.

\bibitem[Sut64]{sutter}
J.~Sutter, \emph{Ein {Beitrag} zum {Typenproblem} der {Riemannschen}
  {Fl\"achen}}, Ph.D. thesis, ETH Z\"urich, 1964.

\bibitem[SZ94]{shabat1}
G.~Shabat and A.~Zvonkin, \emph{Plane trees and algebraic numbers}, Jerusalem
  combinatorics '93, Contemp. Math., vol. 178, Amer. Math. Soc., Providence,
  RI, 1994, pp.~233--275.

\bibitem[Tho92]{thomassen}
C.~Thomassen, \emph{Isoperimetric inequalities and transient random walks on
  graphs}, Ann. Probab. \textbf{20} (1992), no.~3, 1592--1600.

\bibitem[Vol50]{volkovyskii}
L.~I. Volkovyskii, \emph{Investigation of the type problem for a simply
  connected {R}iemann surface surface ({R}ussian)}, Trudy Mat. Inst. Steklov.
  \textbf{34} (1950), 1--171.

\bibitem[Wit55]{wittich1}
H.~Wittich, \emph{Neuere {U}ntersuchungen \"uber eindeutige analytische
  {F}unktionen}, Ergebnisse der Mathematik und ihrer Grenzgebiete (N.F.), Heft
  8, Springer-Verlag, Berlin-G\"ottingen-Heidelberg, 1955.

\end{thebibliography}
\end{document}